\documentclass[reqno,11pt]{amsart}
\usepackage{amssymb}
\usepackage{amsmath}
\usepackage[mathscr]{euscript}
\usepackage{graphicx}
\usepackage{comment}
\usepackage{enumitem}

\usepackage{hyperref}
\RequirePackage[dvipsnames]{xcolor}
\definecolor{halfgray}{gray}{0.55} 
\definecolor{webgreen}{rgb}{0,0.5,0}
\definecolor{webbrown}{rgb}{.6,0,0} 
\hypersetup{
  colorlinks=true, linktocpage=true, pdfstartpage=3,
  pdfstartview=FitV,
  breaklinks=true, pdfpagemode=UseNone, pageanchor=true,
  pdfpagemode=UseOutlines,
  plainpages=false, bookmarksnumbered, bookmarksopen=true,
  bookmarksopenlevel=1,
  hypertexnames=true,
  pdfhighlight=/O,
  urlcolor=webbrown, linkcolor=RoyalBlue,
  citecolor=webgreen,   pdftitle={},
  pdfauthor={},
  pdfsubject={2000 Mathematical Subject Classification: Primary:},%
  pdfkeywords={},
  pdfcreator={pdfLaTeX},
  pdfproducer={LaTeX with hyperref}
}

\usepackage[
  textwidth=15cm,
  textheight=22.7cm,
  centering]{geometry}

\numberwithin{equation}{section}
\theoremstyle{plain}
\newtheorem{theorem}{Theorem}[section]

\newtheorem{corollary}[theorem]{Corollary}
\newtheorem{proposition}[theorem]{Proposition}
\theoremstyle{definition}
\newtheorem{definition}[theorem]{Definition}

\newtheorem{remark}[theorem]{Remark}

\title[Shadowing and Topological Stability in Linear Dynamics]{Shadowing and Topological Stability via Generalized Pseudo-Hyperbolicity in Autonomous and Nonautonomous Linear Dynamics}

\author[D. Dragi\v cevi\'c]{Davor Dragi\v cevi\'c}
\address{Faculty of Mathematics, University of Rijeka, Radmile Matej\v ci\' c 2, 51000 Rijeka, Croatia}
\email{ddragicevic@math.uniri.hr}

\author[M. Pituk]{Mih\'aly Pituk}
\address{Department of Mathematics, University of Pannonia, Egyetem \'ut 10, 8200 Veszpr{\'e}m, Hungary;
HUN--REN--ELTE Numerical Analysis and Large Networks Research Group, Budapest, Hungary}
\email{pituk.mihaly@mik.uni-pannon.hu}



\begin{document}

\begin{abstract}
Although generalized hyperbolicity has proved to be a fundamental notion in linear dynamics, recent results show that it is too restrictive to characterize the shadowing property of invertible bounded linear operators on Banach spaces. This motivates the introduction of generalized pseudo-hyperbolicity, a weaker notion formulated through continuous homogeneous maps generating exponentially decaying Green families. We prove that, for arbitrary sequences of invertible bounded linear operators on Banach spaces, generalized pseudo-hyperbolicity is equivalent to both shadowing and topological stability. In the autonomous setting, this yields the equivalence between shadowing and topological stability, removing additional assumptions from previous results. We also prove robustness of shadowing and topological stability for sequences of invertible bounded linear operators.
\end{abstract}

\maketitle

\section{Introduction}\label{intro}

Generalized hyperbolicity has proved to be a fundamental and highly effective notion in linear dynamics, particularly in the study of shadowing and stability properties of invertible operators on Banach spaces (see, e.g., \cite{B,BCDHP,BCDMP,BM,BM2,BP,CGP,DM,DDM,MTS} and the references therein).
It is known to imply the shadowing property and to entail several strong forms of dynamical stability. However, recent results show that the converse fails in general: there exist invertible operators with the shadowing property which are not generalized hyperbolic; see~\cite{MTS,P}. Thus, despite its strength and usefulness, generalized hyperbolicity is too restrictive to characterize shadowing on arbitrary Banach spaces.

This naturally raises the question of whether there exists a weaker structural notion that captures the shadowing property exactly. A specific form of such a notion is implicit in the proofs of our main results in~\cite{DP}. There, we obtained a spectral characterization of shadowing for invertible bounded linear operators in terms of the surjective spectrum. More precisely, an invertible bounded linear operator~$T$ on a complex Banach space~$X$ has the shadowing property if and only if its surjective spectrum avoids the unit circle in the complex plane. We refer to this condition as \emph{surjective hyperbolicity}.
An important feature of the arguments                          
is that surjective hyperbolicity admits a dynamical formulation which does not rely directly on the spectrum. More precisely, the proofs show that the shadowing property of an invertible operator~$T$ is equivalent to the existence of a bi-infinite sequence of continuous homogeneous maps $C_k\colon X\to X$, $k\in\mathbb Z$, 
and constants $M>0$ and $t\in(0,1)$ such that
\begin{equation*}
C_1-TC_0=I\quad\text{and}\quad
C_{k+1}-TC_k=0\quad\text{whenever $k\ne0$, $k\in\mathbb Z$},
\end{equation*}
and
\begin{equation*}
\|C_k(x)\|\le Mt^{|k|}\|x\|,
\qquad x\in X,\quad k\in\mathbb Z.
\end{equation*}
For details, see Proposition~3.1 and relations~(4.18)--(4.20) in~\cite{DP}. 
Consequently,
\begin{equation*}
G_{m,n}:=C_{m-n},\qquad n,m\in\mathbb Z,
\end{equation*}
is an exponentially decaying Green kernel for the autonomous difference equation
\begin{equation*}
x_{n+1}=Tx_n,\qquad n\in\mathbb Z.
\end{equation*}
Thus, in the autonomous setting, shadowing is encoded by an exponentially decaying Green-kernel structure formed by continuous homogeneous maps.

This observation provides the main motivation for the notion introduced in the present paper. We consider the nonautonomous linear difference equation
\begin{equation*}
x_{n+1}=T_nx_n,\qquad n\in\mathbb Z,
\end{equation*}
where $(T_n)_{n\in\mathbb Z}$ is a sequence of invertible bounded linear operators on a Banach space. In this setting there is no single operator whose surjective spectrum can encode the dynamics, and the translation-invariant representation
$G_{m,n}=C_{m-n}$ is no longer available. The Green-kernel mechanism, however, continues to make sense. This suggests formulating the appropriate nonautonomous analogue in terms of an exponentially decaying two-parameter Green family of continuous homogeneous maps.
Motivated by this structure, we introduce the notion of \emph{generalized pseudo-hyperbolicity}. The use of continuous homogeneous maps, rather than bounded linear projections, is essential. Indeed, surjectivity of a bounded linear operator on a Banach space guarantees the existence of a continuous homogeneous right inverse, whereas a bounded linear right inverse need not exist.

A second objective of the paper is to clarify the relation between shadowing and topological stability. Topological stability of linear operators has been studied in~\cite{BCDHP,LMN}. In the autonomous setting, Bernardes, Caraballo, Darji, F\'avaro and Peris~\cite{BCDHP} proved that every topologically stable invertible operator on a Banach space has the shadowing property; see~\cite[Corollary~15]{BCDHP}. In the converse direction, Nguyen~\cite[Theorem~1.5]{N} proved
that shadowing implies topological stability under the additional
assumption of limit expansiveness. More recently, Messaoudi, Tofanin Neto, Saavedra and Tsokanos~\cite[Theorem~3.5]{MTS} established an equivalence between shadowing, topological stability and strong Lipschitz structural stability for invertible operators whose eigenspaces associated with unimodular eigenvalues admit closed complements.

In this paper, we introduce the natural nonautonomous notion of topological stability, and we prove that, for an arbitrary sequence of invertible bounded linear operators on a Banach space,
\begin{equation*}
\text{shadowing}
\,\Longleftrightarrow\,
\text{generalized pseudo-hyperbolicity}
\,\Longleftrightarrow\,
\text{topological stability}.
\end{equation*}
Consequently, for an arbitrary invertible bounded linear operator on a Banach space,
\begin{equation*}
\text{shadowing}
\,\Longleftrightarrow\,
\text{surjective hyperbolicity}
\,\Longleftrightarrow\,
\text{topological stability}.
\end{equation*}
Thus, our results not only extend the autonomous theory to general nonautonomous linear systems, but also strengthen the existing autonomous theorems by removing the limit expansiveness assumption in~\cite{N} and the complementability assumption in~\cite{MTS}.

Finally, we investigate robustness of shadowing and topological stability for sequences of invertible linear operators $(T_n)_{n\in\mathbb Z}$ on a Banach space.  
 We show that if the unperturbed linear system $(T_n)_{n\in\mathbb Z}$ has the shadowing property and the perturbations $f_n$, $n\in\mathbb Z$, are uniformly small Lipschitz maps, then the perturbed system $(F_n)_{n\in\mathbb Z}$, defined by $
F_n=T_n+f_n$ for $n\in\mathbb Z$,
inherits the shadowing property of the linear system. A similar result is proved for topological stability. 
Both robustness results are based on uniformly bounded solvability of the associated inhomogeneous difference equations, combined with an iterative correction scheme.

The main contributions of the paper can be summarized as follows. 
\vskip5pt
\begin{itemize}
\item The homogeneous Green-kernel structure underlying autonomous surjective hyperbolicity is extended to nonautonomous systems through generalized pseudo-hyperbolicity.
\vskip3pt
\item Generalized pseudo-hyperbolicity is shown to characterize both shadowing and topological stability.
\vskip3pt
\item Shadowing and topological stability are shown to be equivalent for invertible operators on arbitrary Banach spaces.
\vskip3pt
\item Both properties are shown to persist under sufficiently small nonlinear Lipschitz perturbations.	
\end{itemize}
\vskip5pt

The paper is organized as follows. Section~\ref{notation} collects the notation and definitions used throughout the paper. In Section~\ref{mainres}, we formulate the main results, while Section~\ref{proofs} contains their proofs.

\section{Notation and definitions}\label{notation}

Let $\mathbb Z$ and~$\mathbb N$ denote the set of integers and the set of positive integers, respectively. Set $\mathbb N_0:=\mathbb N\cup\{0\}$. 
If $Y_i=(Y_i, \|\cdot \|_i)$, $i=1, 2$, are complex Banach spaces, the symbol 
$H(Y_1, Y_2)$ will denote the space of (possibly nonlinear) maps $F\colon Y_1\to Y_2$ which are continuous and complex-homogeneous, i.e.,
 \begin{equation*}
F(\lambda x)=\lambda F(x)\qquad\text{ for all $x\in Y_1$ and $\lambda \in \mathbb C$}. 
\end{equation*}
It is known that $H(Y_1, Y_2)$ is a complex Banach space when equipped with the norm
\begin{equation*}
\|F\|:=\sup_{x\neq 0}\frac{\|F(x)\|_2}{\|x\|_1},\qquad F\in H(Y_1,Y_2).
\end{equation*}
The space $L(Y_1, Y_2)$ of all bounded linear operators from $Y_1$ to $Y_2$ is a closed subspace of $H(Y_1, Y_2)$. In the particular case when $Y_1=Y_2=Y$ we will write $H(Y)$ and $L(Y)$ instead of $H(Y,Y)$ and $L(Y,Y)$, respectively.

Throughout the paper, we shall assume that~$X$ is a nontrivial complex Banach space. The norms on both $X$ and $H(X)$ will be denoted by $\| \cdot \|$.
By $GL(X)$ we will denote the set of all $T\in L(X)$ which are invertible. Elements of $L(X)$ and~$GL(X)$ will be referred to as \emph{operators} and \emph{invertible operators} on~$X$, respectively. Finally, the symbol $Homeo(X)$ will denote the group of self-homeomorphisms of~$X$. Clearly, $GL(X)$ is a subgroup of $Homeo(X)$.

We first recall the standard notions of shadowing and topological stability.
\begin{definition}\label{autshad}
Let $F\in Homeo(X)$. We say that~$F$ has the \emph{shadowing property} if for every $\epsilon>0$ there exists $\delta>0$ such that every sequence $(x_n)_{n\in\mathbb Z}$ in~$X$ satisfying
\begin{equation*}
\sup_{n\in\mathbb Z}\|x_{n+1}-F(x_n)\|\le \delta
\end{equation*}
is $\epsilon$-shadowed by an exact trajectory of~$F$, i.e., there exists $v\in  X$ such that
\begin{equation*}
\sup_{n\in\mathbb Z}\|x_n-F^n(v)\|\le \epsilon.
\end{equation*}
\end{definition}
As usual, $F^0=I$, $F^n$ denotes the $n$-fold iterate of $F$ for $n>0$, and $F^n=(F^{-1})^{-n}$ for $n<0$. 

For an arbitrary map $F\colon X\to X$, write
\begin{equation*}
\|F\|_\infty:=\sup_{x\in X}\|F(x)\|;\qquad\|F\|_\infty\in[0,\infty].
\end{equation*}
\begin{definition}\label{auttopstab}
Let $F\in Homeo(X)$. We say that~$F$ is \emph{topologically stable} if for every $\epsilon>0$ there exists $\delta>0$ such that, whenever $g\in Homeo(X)$ satisfies
\begin{equation*}
\|F-g\|_\infty\le\delta,
\end{equation*}
there exists a continuous map $h\colon X\to X$ such that
\begin{equation}\label{topconj}
h\circ g=F\circ h
\end{equation}
and
\begin{equation}\label{topclose}
\|h-I\|_\infty\le\epsilon.
\end{equation}
\end{definition}

For $T\in L(X)$, let $\sigma_s(T)$ denote the \emph{surjective spectrum} of~$T$ defined by
\begin{equation}
\sigma_s(T):=\{\lambda\in\mathbb C:\lambda I-T\text{ is not surjective}\}.
\end{equation}
We shall call $T$ \emph{surjectively hyperbolic} if
\begin{equation}\label{surj_hyp}
\sigma_s(T)\cap\mathbb T=\varnothing,
\end{equation}
where $\mathbb T$ denotes the unit circle in the complex plane.
In \cite[Theorem~2.2]{DP}, we proved that, for an arbitrary invertible operator on a Banach space, shadowing is equivalent to surjective hyperbolicity. (For a criterion for Hilbert space operators, see \cite[Theorem~1.7]{PS}.) The argument also yields the continuous-homogeneous structure underlying the generalized pseudo-hyperbolicity property defined below. 

\begin{definition}\label{autgph}
Let $T\in GL(X)$. We say that $T$ is \emph{generalized pseudo-hyperbolic} if there exist $C\in H(X)$ and $D, \lambda>0$ such that for all $n\in \mathbb N_0$,
\begin{equation*}
\|T^n C\| \le De^{-\lambda n} \qquad \text{and} \qquad \|T^{-n}(I-C)\|\le De^{-\lambda n}.
\end{equation*}
\end{definition}
In the particular case where $C\in L(X)$, generalized pseudo-hyperbolicity reduces to pseudo-hyperbolicity, a notion introduced in~\cite{MTS}. While generalized pseudo-hyperbolicity is equivalent to both the shadowing property and topological stability for invertible operators on a Banach space (see Theorem~\ref{thm2}), pseudo-hyperbolicity is only sufficient but not necessary for these properties.

The above definitions serve as prototypes for the corresponding 
nonautonomous notions.
Given a sequence $(h_n)_{n\in\mathbb Z}$ in $Homeo(X)$,
define the associated two-parame\-ter \emph{evolution family} $h\colon\mathbb Z\times\mathbb Z\to Homeo(X)$ by
\begin{equation*}	
h(m,n):=\begin{cases}
h_{m-1}\circ\cdots\circ h_n,\quad &m>n,\\
I,\quad &m=n,\\
h_m^{-1}\circ\cdots\circ h_{n-1}^{-1},\quad &m<n,
\end{cases}
\end{equation*}
where $I$ is the identity map on~$X$.
Clearly, $h$ satisfies the \emph{evolution property}
\begin{equation*}
h(m,k)\circ h(k,n)=h(m,n)\qquad\text{whenever $m,k,n\in\mathbb Z$}.
\end{equation*}

\begin{definition}\label{nonautshad}
Let $(F_n)_{n\in\mathbb Z}$ be a sequence in~$Homeo(X)$. We say that $(F_n)_{n\in\mathbb Z}$ has the \emph{shadowing property} if for every $\epsilon>0$ there exists $\delta>0$ such that every sequence $(x_n)_{n\in\mathbb Z}$ in~$X$ satisfying
\begin{equation*}
\sup_{n\in\mathbb Z}\|x_{n+1}-F_n(x_n)\|\le \delta,
\end{equation*}
there exists $v\in X$ such that 
\begin{equation*}
\sup_{n\in\mathbb Z}\|x_n-F(n,0)(v)\|\le \epsilon,
\end{equation*}
where $F$ is the evolution family associated with $(F_n)_{n\in\mathbb Z}$.
\end{definition}

\begin{definition}\label{nonauttopstab}
Let $(F_n)_{n\in\mathbb Z}$ be a sequence in~$Homeo(X)$. We say that $(F_n)_{n\in\mathbb Z}$ is \emph{topologically stable} if for every $\epsilon>0$ there exists $\delta>0$ such that, whenever $(g_n)_{n\in\mathbb Z}$ is a sequence in $Homeo(X)$ satisfying
\begin{equation*}
\sup_{n\in\mathbb Z}\|F_n-g_n\|_\infty\le\delta,
\end{equation*}
there exists a sequence $(h_n)_{n\in\mathbb Z}$ of continuous self-maps of~$X$ such that
\begin{equation}\label{topconj_2}
h_{n+1}\circ g_n=F_n\circ h_n,\qquad n\in\mathbb Z,
\end{equation}
and
\begin{equation}\label{topclose_2}
\sup_{n\in\mathbb Z}\|h_n-I\|_\infty\le\epsilon.
\end{equation}
\end{definition}

Note that for a sequence $(T_n)_{n\in\mathbb Z}$ in $GL(X)\subset Homeo(X)$ 
 the two-parameter evolution family associated with $(T_n)_{n\in\mathbb Z}$ has the form
\begin{equation*}	
T(m,n):=\begin{cases}
T_{m-1}T_{m-2}\dots T_n,\quad &m>n,\\
I,\quad &m=n,\\
T_m^{-1}T_{m+1}^{-1}\dots T_{n-1}^{-1},\quad &m<n.
\end{cases}
\end{equation*}

The nonautonomous version of generalized pseudo-hyperbolicity is the following.

\begin{definition}\label{gphdef}
Let $(T_n)_{n\in \mathbb Z}$ be a sequence in $GL(X)$. We say that $(T_n)_{n\in \mathbb Z}$ is \emph{generalized pseudo-hyperbolic} if there exist a sequence $(C_n)_{n\in \mathbb Z}$ in $H(X)$ and constants $D, \lambda>0$ such that for all $m\ge n$,
\begin{equation}\label{909}
\|T(m, n)C_n\|\le De^{-\lambda (m-n)} \quad \text{and} \quad \|T(n, m)(I-C_m)\|\le De^{-\lambda (m-n)}.
\end{equation}
\end{definition}

\section{Main results}\label{mainres}

We now state the main results. Their proofs are presented in Section~\ref{proofs}. 

The following theorem shows that generalized pseudo-hyperbolicity characterizes both the shadowing property and topological stability in the general nonautonomous linear setting.

\begin{theorem}\label{thm1}
Let $(T_n)_{n\in\mathbb Z}$ be a sequence in $GL(X)$. Then the following statements are equivalent.
\vskip3pt
\begin{enumerate}[label={(\roman*)},itemindent=1em]
\item[\rm (i)] $(T_n)_{n\in \mathbb Z}$ has the shadowing property.
\vskip2pt
\item[\rm (ii)] $(T_n)_{n\in \mathbb Z}$ is generalized pseudo-hyperbolic.
\vskip2pt
\item[\rm (iii)] $(T_n)_{n\in \mathbb Z}$ is topologically stable. 
\end{enumerate}
\end{theorem}

The following result is the autonomous version of Theorem~\ref{thm1}. 
\begin{theorem}\label{thm2}
Let $T\in GL(X)$. Then the following statements are equivalent.
\vskip3pt
\begin{enumerate}[label={(\roman*)},itemindent=1em]
\item[\rm (i)] $T$ has the shadowing property.
\vskip2pt
\item[\rm (ii)] $T$ is surjectively hyperbolic.
\vskip2pt
\item[\rm (iii)] $T$ is generalized pseudo-hyperbolic.
\vskip2pt
\item[\rm (iv)] $T$ is topologically stable.
\end{enumerate}
\end{theorem}

\begin{remark}\label{rem1}
In \cite[Corollary~15]{BCDHP}, Bernardes et al. proved that topological stability implies the shadowing property for invertible operators on arbitrary Banach spaces.  For the converse implication, Nguyen~\cite[Theorem~1.5]{N} proved that every limit expansive
invertible linear operator with the shadowing property is topologically stable. More recently, Messaoudi et al.~\cite[Theorem~3.5]{MTS}
established the equivalence between topological stability and the shadowing property under the additional assumption that the eigenspaces $\ker(\lambda I-T)$, $\lambda\in\mathbb T$, are complemented in~$X$. Theorem~\ref{thm2} strengthens these results by showing that, for arbitrary invertible linear operators on complex Banach spaces, shadowing and topological stability are equivalent, without assuming either limit expansiveness or complementability of the unimodular eigenspaces.
\end{remark}

By a \emph{Lipschitz self-map of~$X$}, we mean a map 
 $f\colon X\to X$ such that
\begin{equation*}
\operatorname{Lip}(f):=\sup_{x\ne y}\frac{\|f(x)-f(y)\|}{\|x-y\|}<\infty.
\end{equation*}

In the following theorems, we show that both the shadowing property and topological stability of a linear system are robust under small Lipschitz perturbations.
 
\begin{theorem}\label{thm3}
Let $(T_n)_{n\in\mathbb Z}$ be a sequence in $GL(X)$ which has the shadowing property. Then there exists $\eta>0$ such that if $(f_n)_{n\in\mathbb Z}$ is a sequence of Lipschitz self-maps of~$X$ such that 
\begin{equation*}
\operatorname{Lip}(f_n)<\eta\qquad\text{for all $n\in\mathbb Z$}
\end{equation*}
and
\begin{equation*}
F_n:=T_n+f_n\in Homeo(X)\qquad\text{for all $n\in\mathbb Z$},
\end{equation*}
then the sequence $(F_n)_{n\in\mathbb Z}$ has the shadowing property.	
\end{theorem}

\begin{theorem}\label{thm4}
Let $(T_n)_{n\in\mathbb Z}$ be a sequence in $GL(X)$ which is topologically stable. Then there exists $\eta>0$ such that if $(f_n)_{n\in\mathbb Z}$ is a sequence of Lipschitz self-maps of~$X$ such that 
\begin{equation*}
\operatorname{Lip}(f_n)<\eta\qquad\text{for all $n\in\mathbb Z$}
\end{equation*}
and
\begin{equation*}
F_n:=T_n+f_n\in Homeo(X)\qquad\text{for all $n\in\mathbb Z$},
\end{equation*}
then the sequence $(F_n)_{n\in\mathbb Z}$ is topologically stable.	
\end{theorem}

\begin{remark}\label{rem2}
Given $n\in\mathbb Z$, a sufficient condition for $F_n=T_n+f_n$ to be a homeomorphism in Theorems~\ref{thm3} and~\ref{thm4} is that
\begin{equation*}
\|T_n^{-1}\|\operatorname{Lip}(f_n)<1.
\end{equation*}
Indeed, since
\begin{equation*}
\operatorname{Lip}(T_n^{-1}f_n)
\le \|T_n^{-1}\|\operatorname{Lip}(f_n)<1, 
\end{equation*}
it follows that $I+T_n^{-1}f_n$ is a self-homeomorphism of~$X$, and hence so is $F_n=T_n\bigl(I+T_n^{-1}f_n\bigr)$.

\end{remark}

In view of Remark~\ref{rem2}, Theorems~\ref{thm3} and~\ref{thm4} have the following corollaries. 

\begin{corollary}\label{cor1}
Let $(T_n)_{n\in\mathbb Z}$ be a sequence in $GL(X)$ which has the shadowing property and
\begin{equation*}
\sup_{n\in\mathbb Z}\|T_n^{-1}\|<\infty.
\end{equation*}
 Then there exists $\eta>0$ such that if $(f_n)_{n\in\mathbb Z}$ is a sequence of Lipschitz self-maps of~$X$ such that 
\begin{equation*}
\operatorname{Lip}(f_n)<\eta\qquad\text{for all $n\in\mathbb Z$},
\end{equation*}
then 
\begin{equation*}
F_n:=T_n+f_n\in Homeo(X)\qquad\text{for all $n\in\mathbb Z$},
\end{equation*}
and  $(F_n)_{n\in\mathbb Z}$ has the shadowing property.		
\end{corollary}

\begin{corollary}\label{cor2}
Let $(T_n)_{n\in\mathbb Z}$ be a sequence in $GL(X)$ which is topologically stable and
\begin{equation*}
\sup_{n\in\mathbb Z}\|T_n^{-1}\|<\infty.
\end{equation*}
Then there exists $\eta>0$ such that if $(f_n)_{n\in\mathbb Z}$ is a sequence of Lipschitz self-maps of~$X$ such that 
\begin{equation*}
\operatorname{Lip}(f_n)<\eta\qquad\text{for all $n\in\mathbb Z$},
\end{equation*}
then 
\begin{equation*}
F_n:=T_n+f_n\in Homeo(X)\qquad\text{for all $n\in\mathbb Z$},
\end{equation*}
and $(F_n)_{n\in\mathbb Z}$ is topologically stable.	
\end{corollary}

\begin{remark}
The assumption about the uniform boundedness of the inverse operators $T_n^{-1}$, $n\in\mathbb Z$, in Corollaries~\ref{cor1} and~\ref{cor2}, in general,  cannot be omitted. To see this, let
$X=\mathbb C$ and define
\begin{equation*}
T_n (z)=\frac{z}{|n|+2},\qquad z\in\mathbb C,\quad n\in\mathbb Z.
\end{equation*}
It is easily seen that $T_n\in GL(\mathbb C)$ for all $n\in\mathbb Z$ and 
\begin{equation*}
T_n^{-1}(z)=(|n|+2)z,\qquad z\in\mathbb C,\quad n\in\mathbb Z.
\end{equation*}
Since
\begin{equation*}
\|T_j\|=\frac{1}{|j|+2}\leq\frac{1}{2},\qquad j\in\mathbb Z,
\end{equation*}
it follows that
\begin{equation*}
\|T(m,n)\|\leq\prod_{j=n}^{m-1}\frac{1}{|j|+2}\le 2^{-(m-n)},\qquad m\ge n,\quad m,n\in\mathbb Z.
\end{equation*}
This implies that the sequence $(T_n)_{n\in\mathbb Z}$ is generalized pseudo-hyperbolic with $D=1$, $\lambda=\log 2$ and 
$C_n=I$ for all $n\in\mathbb Z$. By Theorem~\ref{thm1}, $(T_n)_{n\in\mathbb Z}$ has the shadowing property and is topologically stable.  Given any $\eta>0$, choose $N\in\mathbb Z$ such that
\begin{equation*}
\|T_N\|=\frac{1}{|N|+2}<\eta,
\end{equation*}
and define
\begin{equation*}
f_N=-T_N\qquad\text{and}\qquad f_n=0\quad\text{whenever $n\ne N$ and $n\in\mathbb Z$}.
\end{equation*}
Clearly,
\begin{equation*}
\operatorname{Lip}(f_n)<\eta
\qquad\text{for all }n\in\mathbb Z,
\end{equation*}
whereas $F_N=T_N+f_N=0$. Thus, $F_N\notin Homeo(\mathbb C)$. Of course, in this example, we have
\begin{equation*}
\sup_{n\in\mathbb Z}\|T_n^{-1}\|=\infty.
\end{equation*}
\end{remark}

Finally, we present the analogues of the robustness theorems in the autonomous case.

\begin{theorem}\label{thm5}
Let $T\in GL(X)$ with the shadowing property. Then there exists $\eta>0$ such that if $f$ is a Lipschitz self-map of~$X$ with
\begin{equation*}
\operatorname{Lip}(f)<\eta,\end{equation*}
then
\begin{equation*}
F:=T+f\in Homeo(X),
\end{equation*}
and $F$ has the shadowing property.	
\end{theorem}

\begin{theorem}\label{thm6}
Let $T\in GL(X)$ be topologically stable. Then there exists $\eta>0$ such that if $f$ is a Lipschitz self-map of~$X$ with
\begin{equation*}
\operatorname{Lip}(f)<\eta,\end{equation*}
then
\begin{equation*}
F:=T+f\in Homeo(X),
\end{equation*}
and $F$ is topologically stable.	
\end{theorem}

\section{Proof of the main results}\label{proofs}

As a preparation for the proof of the main theorems, we establish some preliminary results.

 Let $\ell_\infty(X)=\ell_\infty(\mathbb Z, X)$ denote the Banach space of all bounded sequences $x=(x_n)_{n\in\mathbb Z}$ in~$X$, equipped with the norm
\begin{equation*}
\|x\|_\infty:=\sup_{n\in\mathbb Z}\|x_n\|,\qquad x\in\ell_\infty(X).
\end{equation*}
Given a sequence $(T_n)_{n\in\mathbb Z}$ in $GL(X)$, consider the difference operator 
\begin{equation}\label{Ldef}
(\mathcal Lx)_n:=x_{n+1}-T_n x_n,\qquad n\in\mathbb Z.
\end{equation}
Since we do not assume the uniform boundedness of operators $T_n$, $n\in\mathbb Z$, $\mathcal L$ need not act boundedly on all of $\ell_\infty(X)$. We shall therefore consider~$\mathcal L$ on its natural domain (see~\cite[p.~230]{H})
\begin{equation*}
D(\mathcal L)=\ell'_\infty(X):=\{\,x\in\ell_\infty(X):\mathcal Lx\in\ell_\infty(X)\,\}\subset\ell_\infty(X).
\end{equation*}
\begin{proposition}\label{prop1}
$\mathcal L\colon D(\mathcal L)\subset\ell_\infty(X)\to\ell_\infty(X)$ is a closed linear operator.
\end{proposition}
\begin{proof}
It is easy to verify that $\mathcal L$ is linear. We now prove that it is closed.
Let $(x^k)_{k\in \mathbb N}$ be a sequence in $D(\mathcal L)$ such that $x^k \to x$ and $\mathcal Lx^k \to y$ in $\ell_\infty(X)$. Writing $y^k:=\mathcal L x^k$, for each $n\in\mathbb Z$, we have \begin{equation*}
y_n=\lim_{k\to \infty}y_n^k=\lim_{k\to \infty}(x_{n+1}^k-T_n x_n^k)=\lim_{k\to \infty}x_{n+1}^k-T_n \lim_{k\to \infty}x_n^k=x_{n+1}-T_n x_n 
\end{equation*}
by the continuity of~$T_n$. Thus, $\mathcal L x=y$. Since $y\in\ell_\infty(X)$, it follows that $x\in D(\mathcal L)$. Hence $\mathcal L$ is closed.
\end{proof}
We equip $D(\mathcal L)=\ell'_\infty(X)$ with the graph norm
\begin{equation*}
\|x\|'_\infty:=\|x\|_\infty+\|\mathcal L x\|_\infty ,\qquad x\in\ell'_\infty(X).
\end{equation*}
Since $\mathcal L$ is closed, $(\ell'_\infty(X),\|\cdot\|'_\infty)$ is a Banach space. Moreover,
\begin{equation*}
\|\mathcal L x\|_\infty\leq\|x\|'_\infty,\qquad x\in\ell'_\infty(X),
\end{equation*} 
 and hence
$
\mathcal L:(\ell_\infty'(X),\|\cdot\|_\infty')\longrightarrow(\ell_\infty(X),\|\cdot\|_\infty)
$
is a bounded linear operator.

The following simple result is implicitly established in~\cite{BD}. We present its proof only for the sake of completeness. 
\begin{proposition}\label{prop2}
Let $(T_n)_{n\in \mathbb Z}$ be a sequence in $GL(X)$. Then the following statements are equivalent.
\vskip3pt
\begin{enumerate}[label={(\roman*)},itemindent=1em]
\item[\rm (i)] $(T_n)_{n\in \mathbb Z}$ has the shadowing property.
\vskip2pt
\item[\rm (ii)] $\mathcal L:(\ell_\infty'(X),\|\cdot\|_\infty')\longrightarrow(\ell_\infty(X),\|\cdot\|_\infty)$ is surjective.
\end{enumerate}
\end{proposition}

\begin{proof}
Suppose that $(T_n)_{n\in \mathbb Z}$ has the shadowing property. Let $\delta>0$ be as in Definition~\ref{nonautshad} corresponding to $\epsilon=1$. Let $y\in \ell_\infty(X)$ be an arbitrary nonzero sequence. (The case $y=0$ is trivial.) Choose $x_0\in X$ arbitrarily. Since each $T_n$, $n\in\mathbb Z$, is invertible, the recurrence
\begin{equation*}
x_{n+1}-T_n x_n=\frac{\delta}{\|y\|_\infty}y_n, \qquad n\in \mathbb Z,
\end{equation*}
determines a unique bi-infinite sequence $(x_n)_{n\in\mathbb Z}$ in~$X$ with the prescribed~$x_0$.
Then 
\begin{equation*}
\sup_{n\in \mathbb Z}\|x_{n+1}-T_n x_n\|=\delta.
\end{equation*}
By the shadowing property,  there exists $v\in X$ such that
 \begin{equation*}
 \sup_{n\in \mathbb Z}\|x_n-T(n, 0)v\|\le 1.
\end{equation*}
Set
\begin{equation*}
z_n:=\frac{\|y\|_\infty}{\delta}(x_n-T(n, 0)v), \qquad n\in \mathbb Z.
\end{equation*}
Then $z=(z_n)_{n\in\mathbb Z}\in\ell_\infty(X)$. Moreover, using
\begin{equation*}
T(n+1,0)=T_nT(n,0),\qquad n\in\mathbb Z,
\end{equation*}
we obtain
\begin{equation*}
\begin{aligned}
(\mathcal Lz)_n
&=z_{n+1}-T_nz_n\\
&=\frac{\|y\|_\infty}{\delta}
\Bigl(
x_{n+1}-T(n+1,0)v
-T_nx_n+T_nT(n,0)v
\Bigr)\\
&=\frac{\|y\|_\infty}{\delta}
\bigl(x_{n+1}-T_nx_n\bigr)=y_n.
\end{aligned}
\end{equation*}
Thus, $\mathcal Lz=y$. Since $z\in\ell_\infty(X)$ and
$\mathcal L z=y\in\ell_\infty(X)$, it follows that $z\in \ell'_\infty(X)$.
Hence $\mathcal L$ is surjective.

Suppose now that $\mathcal L$ is surjective. By the Open Mapping Theorem (see~\cite[Theorem 2.12.1]{HP}), there exists $K>0$ such that, for each $z\in \ell_\infty(X)$, there exists $y\in \ell_\infty'(X)$ such that 
\begin{equation}\label{259}
\mathcal Ly=z \quad \text{and} \quad 
 \|y\|_\infty'\le K\|z\|_\infty.\end{equation} Given $\epsilon>0$, set $\delta:=\frac{\epsilon}{K}$. Let $(x_n)_{n\in \mathbb Z}$ be an arbitrary sequence in~$X$ such that 
\begin{equation*}    
\sup_{n\in \mathbb Z}\|x_{n+1}-T_n x_n\| \le \delta.
\end{equation*}
Set
\begin{equation}\label{591}
z_n:=x_{n+1}-T_n x_n, \qquad n\in \mathbb Z.
\end{equation}
Then $z=(z_n)_{n\in \mathbb Z}\in \ell_\infty(X)$ with $\|z\|_\infty \le \delta$. Choose $y=(y_n)_{n\in \mathbb Z}\in \ell_\infty'(X)$ such that~\eqref{259} holds. From $\mathcal L y=z$ and~\eqref{591}, we obtain\begin{equation*}
y_{n+1}-T_n y_n=z_n=x_{n+1}-T_n x_n,\qquad n\in \mathbb Z.
\end{equation*}
Hence
\begin{equation*}
x_{n+1}-y_{n+1}=T_n(x_n-y_n), \qquad n\in \mathbb Z,
\end{equation*}
which implies that 
\begin{equation}\label{306}
x_n-y_n=T(n,0)(x_0-y_0),\qquad n\in\mathbb Z.
\end{equation}
Set $v:=x_0-y_0$. By~\eqref{259} and~\eqref{306}, we have for $n\in\mathbb Z$,
\begin{equation*}
\|x_n-T(n, 0)v\|=\|x_n-(x_n-y_n)\|=\|y_n\|\le\|y\|_\infty \le \|y\|'_\infty \le K\|z\|_\infty \le K\delta=\epsilon.
\end{equation*}
Thus, $(T_n)_{n\in \mathbb Z}$ has the shadowing property.
\end{proof}

We shall also need the following homogeneous right-inverse perturbation result.
Its proof is analogous to its special case in ~\cite[Proposition 3.12]{DP}.

\begin{proposition}\label{prop3}
Let $Y_i$, $i=1, 2$, be complex Banach spaces, and let $T\in L(Y_1, Y_2)$ be surjective. Then there exists 
$R_T\in H(Y_2, Y_1)$ such that 
\begin{equation}\label{617}
TR_T=I_{Y_2}.
\end{equation}
 Moreover, if $R_T\in H(Y_2, Y_1)$ satisfies~\eqref{617} and 
\begin{equation*}
0<\alpha <\frac 1 K, \qquad \text{where $K:=\|R_T\|$,}
\end{equation*}
then, for every $S\in L(Y_1, Y_2)$ with $\|T-S\| \le \alpha$, there exists $R_S\in H(Y_2, Y_1)$ such that
\begin{equation}\label{440}
SR_S=I_{Y_2}\qquad\text{and}\qquad \|R_S\| \le \frac{K}{1-\alpha K}.
\end{equation}
\end{proposition}
\begin{proof}
The existence of $R_T\in H(Y_2, Y_1)$ satisfying~\eqref{617} follows from \cite[Chap.~IV, Sec.~27, Proposition~5~(i)]{Mul}. Define a sequence of functions $F_n\colon Y_2\rightarrow Y_2$, $n\in\mathbb N$, recursively by $F_1(y)=y$ for $y\in Y_2$ and  
\begin{equation}\label{xndef}
	F_{n+1}(y)=(T-S)(R_T(F_n(y)))\qquad\text{for $y\in Y_2$ and $n\in\mathbb N$}.
\end{equation}
It follows by easy induction that the functions~$F_n$, $n\in\mathbb N$, are continuous and homogeneous, i.e., $F_n\in H(Y_2)$ for $n\in\mathbb N$. It follows that 
\begin{equation*}
	\|F_{n+1}(y)\|\leq\|T-S\|\|R_T\|\|F_n(y)\|\leq K\alpha\|F_n(y)\|\qquad\text{for $y\in Y_2$ and $n\in\mathbb N$}. 
\end{equation*}
Hence,
\begin{equation}\label{xnest}
	\|F_n(y)\|\leq (K\alpha)^{n-1}\|y\|\qquad\text{for all $y\in Y_2$ and $n\in\mathbb N$}.
\end{equation}
This yields
\begin{equation}\label{RTest}
	\|R_T(F_n(y))\|\leq \|R_T\|\|F_n(y)\|\leq K(K\alpha)^{n-1}\|y\|\qquad\text{for all $y\in Y_2$ and $n\in\mathbb N$}.
\end{equation}
Since $0< K\alpha<1$, we can define $R_S\colon Y_2 \to Y_1$ by
\begin{equation}\label{RSdef}
	R_S(y)=\sum_{n=1}^\infty R_T(F_n(y))\qquad\text{for $y\in Y_2$}.
\end{equation}
Since $R_T\in H(Y_2, Y_1)$ and $F_n\in H(Y_2)$, $n\in\mathbb N$,  we have that $R_S$ is homogeneous. From~\eqref{RTest}, it follows that, on each bounded subset of~$Y_2$, the convergence of the series in~\eqref{RSdef} is uniform. By the Weierstrass M-test, $R_S$ is continuous and hence $R_S\in H(Y_2, Y_1)$.  Moreover, from~\eqref{RTest} and~\eqref{RSdef}, we find for $y\in Y_2$,
\begin{equation*}
	\|R_S(y)\|\leq K\sum_{n=1}^\infty(K\alpha)^{n-1} \|y\|=\frac{K}{1-K\alpha}\|y\|,
\end{equation*}
which proves the norm estimate in~\eqref{440}. Next, the identity 
\begin{equation*}
	(T-S)R_T=TR_T-SR_T=I_{Y_2}-SR_T,
\end{equation*}
together with~\eqref{xndef}, implies for $y\in Y_2$ and $n\in\mathbb N$,
\begin{equation*}
	F_{n+1}(y)=F_n(y)-S(R_T(F_n(y))).
\end{equation*}
Hence
\begin{equation*}
	S(R_T(F_n(y)))=F_n(y)-F_{n+1}(y)\qquad\text{for all $y\in Y_2$ and $n\in\mathbb N$}.
\end{equation*}
From this and~\eqref{RSdef}, we conclude that for all $y\in Y_2$,
\begin{equation*}
	S(R_S(y))=\sum_{n=1}^\infty S(R_T(F_n(y)))=\sum_{n=1}^\infty(F_n(y)-F_{n+1}(y))=F_1(y)=y.
\end{equation*}
	Thus, \eqref{440} holds.
	\end{proof}
	
	Now we can give proofs of the main results stated in Section~\ref{mainres}.
	
	\begin{proof}[Proof of Theorem~\ref{thm1}]
	\textit{(i) $\Rightarrow$ (ii).}
	Suppose that $(T_n)_{n\in \mathbb Z}$ has the shadowing property. Given $z\ge 1$ and $n\in \mathbb Z$, we define  $\mathcal L_{z, n}\colon \ell_\infty'(X)\to \ell_\infty(X)$ by 
\begin{equation*}
(\mathcal L_{z, n}x)_m=\begin{cases}
\frac{1}{z}x_{m+1}-T_m x_m & m\ge n; \\
zx_{m+1}-T_m x_m & m<n,
\end{cases} \quad \text{for $x=(x_m)_{m\in \mathbb Z}\in \ell_\infty'(X)$,}
\end{equation*}
where $\ell'_\infty(X)$ is equipped with the graph norm $\|\cdot\|'_\infty$. It is easily seen that 
\begin{equation*}
\mathcal L_{z,n}\in L(\ell'_\infty(X),\ell_\infty(X)),\qquad z\geq1,\,\,n\in\mathbb Z.
\end{equation*}
Since $|\frac 1 z-1| \le z-1$ for $z\ge 1$, we have for $x\in\ell'_\infty(X)$ and $m\in\mathbb Z$,
\begin{equation*}
\|(\mathcal L_{z, n}x)_m-(\mathcal Lx)_m\| \le (z-1)\|x_{m+1}\|,
\end{equation*}
where $\mathcal L\in L(\ell'_\infty(X),\ell_\infty(X))$ is defined by~\eqref{Ldef}.
Hence
\begin{equation*}
\|\mathcal L_{z, n}x-\mathcal L x\|_\infty \le (z-1)\|x\|_\infty \le (z-1)\|x\|'_\infty
\end{equation*}
whenever $z\geq1$, $n\in\mathbb Z$ and $x\in\ell'_\infty(X)$. Consequently,
\begin{equation}\label{closest}
\|\mathcal L_{z, n}-\mathcal L \| \le z-1,\qquad z\geq1,\,\, n\in\mathbb Z.
\end{equation}
By Proposition~\ref{prop2}, the shadowing property of $(T_n)_{n\in\mathbb Z}$ implies that $\mathcal L$ is surjective.
From this, by Proposition~\ref{prop3}, we conclude that $\mathcal L$ has a right-inverse $\mathcal R\in H(\ell_\infty(X), \ell_\infty'(X))$. Set $K:=\|\mathcal R\|$.  If $z\in (1, \frac{K+1}{K})$, then $0<z-1<\frac{1}{K}$. This, together with~\eqref{closest}, implies that
\begin{equation*}
\|\mathcal L_{z, n}-\mathcal L \| \le z-1<\frac{1}{K}\quad\text{whenever $z\in\biggl(1,\frac{K+1}{K}\biggr)$ and $n\in\mathbb Z$}.
\end{equation*}
It follows from the second part of Proposition~\ref{prop3}  that, for every $z\in (1, \frac{K+1}{K})$ and $n\in \mathbb Z$, $\mathcal L_{z, n}$ has a right-inverse $\mathcal R_{z, n}\in H(\ell_\infty(X), \ell_\infty'(X))$ satisfying 
\begin{equation*}
\|\mathcal R_{z, n}\| \le \frac{K}{1-(z-1)K}.
\end{equation*}
Fix $t\in (\frac{K}{K+1}, 1)$ and $n\in \mathbb Z$. Then $z:=1/t\in(1,\frac{K+1}{K})$. Therefore, $\mathcal L_{1/t, n}$ has a right-inverse $\mathcal R_{1/t, n}\in H(\ell_\infty(X), \ell_\infty'(X))$ satisfying
\begin{equation}\label{Rnormest}
\|\mathcal R_{1/t, n}\| \le M,\qquad\text{where $
M:=\frac{K}{1-(1/t-1)K}>0$}.
\end{equation}
We construct a sequence of maps $G_{m, n}\in H(X)$, $m\in \mathbb Z$, with the following properties:
\begin{equation}\label{900}
\|G_{m, n}(v)\|\le Mt^{|m-n|}\|v\| \quad \text{whenever $v\in X$ and $m\in \mathbb Z$,}
\end{equation}
\begin{equation}\label{906}
G_{n+1, n}-T_n G_{n, n}=I,
\end{equation}
and
\begin{equation}\label{907}
G_{m+1, n}-T_m G_{m, n}=0 \quad\text{whenever $m\in\mathbb Z$ and $m\neq n$}.
\end{equation}
Let $v\in X$. Define $y(v)=(y_m(v))_{m\in \mathbb Z}\in \ell_\infty(X)$ by 
\begin{equation*}
y_m(v)=\begin{cases}
v,\quad &m=n; \\
0,\quad & m\neq n.
\end{cases}
\end{equation*}
For $m\in\mathbb Z$, define $G_{m,n}\colon X\to X$ by
\begin{equation*}
G_{m,n}(v):=t^{|m-n|}\bigl(\mathcal R_{1/t,n}(y(v))\bigr)_m,\qquad v\in X.
\end{equation*}
The map $v\mapsto y(v)$ from~$X$ to $\ell_\infty(X)$ is continuous and homogeneous, $\mathcal R_{1/t,n}$ is continuous and homogeneous and the coordinate evaluation $x\mapsto x_m$, $m\in\mathbb Z$, from $\ell'_\infty(X)$ to~$X$ is linear and continuous. Thus, $G_{m,n}\in H(X)$ for each $m\in\mathbb Z$, as a composition of continuous and homogeneous maps. Clearly, $\|y(v)\|_\infty=\|v\|$ for $v\in X$. This, together with~\eqref{Rnormest}, implies that~\eqref{900} holds. Writing $y=y(v)$ and $x=\mathcal R_{1/t,n}y$, we have $\mathcal L_{1/t,n} x=y$. From this and the definition of $\mathcal L_{1/t,n} x$, we find for $v\in X$ and $m\geq n$,
\begin{equation*}
\begin{aligned}
G_{m+1,n}(v)-T_mG_{m,n}(v)
&=t^{|m+1-n|}x_{m+1}-t^{|m-n|}T_m x_m\\
&=t^{m-n}\bigl(tx_{m+1}-T_mx_m\bigr)\\
&=t^{m-n}\bigl(\mathcal L_{1/t,n}x\bigr)_m\\
&=t^{m-n}y_m,
\end{aligned}
\end{equation*}
while, for $m<n$,
\begin{equation*}
\begin{aligned}
G_{m+1,n}(v)-T_mG_{m,n}(v)
&=t^{|m+1-n|}x_{m+1}-t^{|m-n|}T_m x_m\\
&=t^{n-m}\left(\frac1t x_{m+1}-T_mx_m\right)\\
&=t^{n-m}\bigl(\mathcal L_{1/t,n}x\bigr)_m\\
&=t^{n-m}y_m.
\end{aligned}
\end{equation*}
From this and from the definition of $y=y(v)$, it follows for $v\in X$ and $m\in\mathbb Z$,
\begin{equation*}
G_{m+1,n}(v)-T_m G_{m,n}(v)=\begin{cases}
v,\quad & m=n; \\
0,\quad & m\neq n.
\end{cases}
\end{equation*}
This implies both~\eqref{906} and~\eqref{907}.

Now set
\begin{equation*}
C_n:=G_{n,n-1}\in H(X),\qquad n\in\mathbb Z.
\end{equation*}
From~\eqref{906} and~\eqref{907}, it follows easily by forward and backward induction that
\begin{equation}\label{greenformula}
G_{m,n-1}=
\begin{cases}
T(m,n)C_n, & m\ge n,\\[1mm]
-T(m,n)(I-C_n), & m<n.
\end{cases}
\end{equation}
This, combined with~\eqref{900}, yields for $m\geq n$,
 \begin{equation*}
\|T(m,n)C_n\|=\|G_{m,n-1}\|
\le M t^{m-n+1}\leq\frac{M}{t}t^{m-n}.
\end{equation*}
On the other hand, for $m<n$, we have
\begin{equation*}
\|T(m,n)(I-C_n)\|
=\|G_{m,n-1}\|
\le M t^{|m-n+1|}
=\frac{M}{t}t^{n-m}.
\end{equation*}
Equivalently, after interchanging~$m$ and $n$, we obtain for $m>n$,
\begin{equation*}
\|T(n,m)(I-C_m)\|
\leq\frac{M}{t}t^{m-n}.
\end{equation*}
It remains to  consider the second estimate in the case $m=n$. By~\eqref{900}, 
\begin{equation*}
\|C_n\|=\|G_{n,n-1}\|\leq M t,\qquad n\in\mathbb Z.
\end{equation*}
Hence, 
\begin{equation*}
\|I-C_n\|\le 1+\|C_n\|\leq 1+Mt,\qquad n\in\mathbb Z.
\end{equation*}
Therefore, the exponential estimates in Definition~\ref{gphdef} are satisfied with
\begin{equation*}
D:=\max\biggl\{\frac{M}{t},1+Mt\biggr\}>0\qquad\text{and}\qquad \lambda:=-\log t>0.
\end{equation*}
Thus, $(T_n)_{n\in\mathbb Z}$ is generalized pseudo-hyperbolic, and the proof of $(i)\Rightarrow(ii)$ is complete.
\medskip

\textit{(ii) $\Rightarrow$ (iii).}
Suppose that $(T_n)_{n\in\mathbb Z}$ is generalized pseudo-hyperbolic. Let 
$C_n\in H(X)$, $n\in\mathbb Z$, and constants $D,\lambda>0$ have the meaning from Definition~\ref{gphdef}.
Define the Green family $G\colon \mathbb Z\times\mathbb Z\to H(X)$ by 
\begin{equation}\label{Gmndef}
G_{m,n}=
\begin{cases}
T(m,n+1)C_{n+1},\quad & m\geq n+1,\\[1mm]
-T(m,n+1)(I-C_{n+1}),\quad & m<n+1.
\end{cases}
\end{equation}
From \eqref{Gmndef}, it follows that
\begin{equation}\label{greenident}
G_{m+1,n}-T_mG_{m,n}=
\begin{cases}
I,\quad & m=n,\\
0,\quad & m\ne n.
\end{cases}
\end{equation}
Indeed, if $m\geq n+1$, then  $m+1\geq n+1$, and by the upper equality in~\eqref{Gmndef},
\begin{equation*}
G_{m+1,n}
=T(m+1,n+1)C_{n+1}
=T_mT(m,n+1)C_{n+1}
=T_mG_{m,n}.
\end{equation*}
Hence
\begin{equation*}
G_{m+1,n}-T_m G_{m,n}=0,\qquad m\geq n+1.
\end{equation*}
If $m<n$, then $m<n+1$ and $m+1<n+1$, and by the lower equality in~\eqref{Gmndef},
\begin{equation*}
\begin{aligned}
G_{m+1,n}
&=-T(m+1,n+1)(I-C_{n+1})\\
&=-T_mT(m,n+1)(I-C_{n+1})
=T_mG_{m,n}.
\end{aligned}
\end{equation*}
Hence
\begin{equation*}
G_{m+1,n}-T_m G_{m,n}=0,\qquad m<n.
\end{equation*}
Finally, if $m=n$, then
\begin{equation*}
G_{n+1,n}=C_{n+1},
\end{equation*}
while
\begin{equation*}
T_nG_{n,n}
=-T_nT(n,n+1)(I-C_{n+1})
=-(I-C_{n+1}).
\end{equation*}
Hence
\begin{equation*}
G_{n+1,n}-T_nG_{n,n}
=C_{n+1}+I-C_{n+1}=I.
\end{equation*}
Thus, \eqref{greenident} holds.
The exponential estimates from Definition~\ref{gphdef} imply that
\begin{equation}\label{greenesti}
\|G_{m,n}\|\le De^{-\lambda|m-n-1|},
\qquad m,n\in\mathbb Z.
\end{equation}
Set
\begin{equation*}
\mu:=D\sum_{j\in\mathbb Z}e^{-\lambda|j|}
=D\frac{1+e^{-\lambda}}{1-e^{-\lambda}}>0.
\end{equation*}
Let $\epsilon>0$. Choose $\delta:=\epsilon/\mu$. Let $(g_n)_{n\in\mathbb Z}$ be  an arbitrary sequence in $Homeo(X)$ satisfying
\begin{equation*}
\sup_{n\in\mathbb Z}\|T_n-g_n\|_\infty\le\delta.
\end{equation*}
Put
\begin{equation}\label{fndef}
f_n:=T_n-g_n,\qquad n\in\mathbb Z.
\end{equation}
Then 
\begin{equation}\label{fnest}
\sup_{x\in X}\|f_n(x)\|=\sup_{x\in X}\|T_n(x)-g_n(x)\|=\|T_n-g_n\|_\infty\leq\delta,\qquad n\in\mathbb Z.
\end{equation}
For $n\in\mathbb Z$ and $x\in X$, define
\begin{equation}\label{Hndef}
H_n(x):=\sum_{m\in\mathbb Z}
G_{n,m}\bigl(f_m(g(m,n)(x))\bigr).
\end{equation}
By \eqref{greenesti} and~\eqref{fnest}, we have for $n,m\in\mathbb Z$,
\begin{equation*}
\begin{aligned}
\|G_{n,m}(f_m(g(m,n)(x)))\|
&\le De^{-\lambda|n-m-1|}\,
   \|f_m(g(m,n)(x))\|\\
&\le D\delta e^{-\lambda|n-m-1|}.
\end{aligned}
\end{equation*}
Since
\begin{equation*}
\sum_{m\in\mathbb Z}e^{-\lambda|n-m-1|}
=\sum_{j\in\mathbb Z}e^{-\lambda|j|}<\infty,\qquad n\in\mathbb Z,
\end{equation*}
the series in \eqref{Hndef} converges uniformly in $x\in X$ for every $n\in\mathbb Z$. Hence, each $H_n$, $n\in\mathbb Z$, is continuous. Furthermore,
\begin{equation*}
\begin{aligned}
\|H_n(x)\|
\le D\delta\sum_{m\in\mathbb Z}e^{-\lambda|n-m-1|}=D\delta \sum_{j\in\mathbb Z}e^{-\lambda|j|}=\mu
\delta=\epsilon
\end{aligned}
\end{equation*}
for $n\in\mathbb Z$ and $x\in X$.
Hence
\begin{equation}\label{eq:hbound}
\sup_{n\in\mathbb Z}\|H_n\|_\infty\le\epsilon.
\end{equation}
We now verify the semiconjugacy relation. By the evolution property of the family associated with $(g_n)_{n\in\mathbb Z}$,
\begin{equation*}
g(m,n+1)\circ g_n=g(m,n),
\qquad m,n\in\mathbb Z.
\end{equation*}
Consequently, using \eqref{Hndef},
\begin{equation*}
\begin{aligned}
H_{n+1}(g_n(x))-T_n H_n(x)
&=\sum_{m\in\mathbb Z}
G_{n+1,m}\bigl(f_m(g(m,n+1)(g_n(x)))\bigr)\\
&\qquad-T_n\sum_{m\in\mathbb Z}
G_{n,m}\bigl(f_m(g(m,n)(x))\bigr)\\
&=\sum_{m\in\mathbb Z}
\bigl(G_{n+1,m}-T_nG_{n,m}\bigr)
\bigl(f_m(g(m,n)(x))\bigr).
\end{aligned}
\end{equation*}
From~\eqref{greenident}, it follows that the only nonzero term in the last sum corresponds to $m=n$, and since $g(n,n)=I$, we have for $n\in\mathbb Z$ and $x\in X$,
\begin{equation}\label{Hncorr}
H_{n+1}(g_n(x))-T_n H_n(x)=f_n(x).
\end{equation}
Define
\begin{equation*}
h_n:=I+H_n,\qquad n\in\mathbb Z.
\end{equation*}
Since $f_n=T_n-g_n$, relation \eqref{Hncorr} gives
\begin{equation*}
\begin{aligned}
h_{n+1}(g_n(x))
&=g_n(x)+H_{n+1}(g_n(x))\\
&=g_n(x)+T_n H_n(x)+f_n(x)\\
&=g_n(x)+T_n H_n(x)+T_nx-g_n(x)\\
&=T_n H_n(x)+T_n x\\
&=T_n h_n(x).
\end{aligned}
\end{equation*}
Hence
\begin{equation}\label{topc1}	
h_{n+1}\circ g_n=T_n\circ h_n,
\qquad n\in\mathbb Z.
\end{equation}

Finally, by \eqref{eq:hbound},
\begin{equation}\label{topc2}	
\sup_{n\in\mathbb Z}\|h_n-I\|_\infty
=\sup_{n\in\mathbb Z}\|H_n\|_\infty
\le\epsilon.
\end{equation}

Thus, $(T_n)_{n\in\mathbb Z}$ is topologically stable. This proves $(ii)\Rightarrow(iii)$.
\medskip

\textit{(iii) $\Rightarrow$ (i).}
Suppose that $(T_n)_{n\in\mathbb Z}$ is topologically stable. Let $\epsilon>0$ and choose $\delta>0$ from Definition~\ref{nonauttopstab}. Let $(x_n)_{n\in\mathbb Z}$ be an arbitrary sequence in~$X$ satisfying
\begin{equation}\label{pseudo}
\sup_{n\in\mathbb Z}\|x_{n+1}-T_nx_n\|\le\delta.
\end{equation}
Define
\begin{equation*}
g_n(x):=T_nx+d_n,\qquad n\in\mathbb Z,
\end{equation*}
where $d_n:=x_{n+1}-T_nx_n$ for $n\in\mathbb Z$.
Since $T_n\in GL(X)$, each $g_n$, $n\in\mathbb Z$, is a homeomorphism, and
\begin{equation*}
\|g_n-T_n\|_\infty=\|d_n\|\le\delta
\end{equation*}
by~\eqref{pseudo}. By the choice of~$\delta$,
there exists a sequence $(h_n)_{n\in\mathbb Z}$ of continuous self-maps of~$X$ such that
\begin{equation}\label{topconjmod}
h_{n+1}\circ g_n=T_n\circ h_n,\qquad n\in\mathbb Z,
\end{equation}
and
\begin{equation}\label{topclosemod}
\sup_{n\in\mathbb Z}\|h_n-I\|_\infty\le\epsilon.
\end{equation}
 Since $g_n(x_n)=x_{n+1}$ for $n\in\mathbb Z$, \eqref{topconjmod} implies that
\begin{equation*}
h_{n+1}(x_{n+1})=T_n h_n(x_n),\qquad n\in\mathbb Z.
\end{equation*}
From this, it follows by forward and backward induction that
\begin{equation*}
h_n(x_n)=T(n,0)h_0(x_0),\qquad n\in\mathbb Z.
\end{equation*}
This, together with~\eqref{topclosemod}, implies that if we define $v:=h_0(x_0)$, then\begin{equation*}
\|x_n-T(n,0)v\|=\|x_n-h_n(x_n)\|\le\epsilon, \qquad n\in\mathbb Z.
\end{equation*}
Thus, $(T_n)_{n\in\mathbb Z}$ has the shadowing property.
		\end{proof}
	
	\begin{proof}[Proof of Theorem~\ref{thm2}]
	
	\textit{(i) $\Leftrightarrow$ (ii).}
By \cite[Theorem~2.2]{DP}, $T\in GL(X)$ has the shadowing property if and only if it is surjectively hyperbolic.

\noindent
\textit{(i) $\Rightarrow$ (iii).}
Suppose that $T\in GL(X)$ has the shadowing property.  Consider the constant sequence $(T_n)_{n\in\mathbb Z}$ in $GL(X)$, where $T_n=T$ for all $n\in\mathbb Z$. Since the corresponding evolution family is given by
\begin{equation}\label{autevol}
T(m,n)=T^{m-n},\qquad m,\,n\in\mathbb Z,
\end{equation}
it follows from Definitions~\ref{autshad} and~\ref{nonautshad} that the shadowing property of~$T$ is equivalent to the shadowing property of $(T_n)_{n\in\mathbb Z}$. 
Therefore, by Theorem~\ref{thm1}, this constant sequence $(T_n)_{n\in\mathbb Z}$ is generalized pseudo-hyperbolic. Thus, there exist
a sequence $(C_n)_{n\in\mathbb Z}$ in $H(X)$ and constants $D,\lambda>0$ such that,
for every $m\ge n$,
\begin{equation*}
\|T^{m-n}C_n\|\le De^{-\lambda(m-n)}
\end{equation*}
and
\begin{equation*}
\|T^{n-m}(I-C_m)\|\le De^{-\lambda(m-n)}.
\end{equation*}
Taking $m=k$ and $n=0$ in the first inequality and $m=0$ and $n=-k$ in the second inequality, it follows that the exponential estimates in Definition~\ref{autgph} are satisfied with $C:=C_0$. Thus, $T$ is generalized pseudo-hyperbolic.
\medskip

\textit{(iii) $\Rightarrow$ (iv).}
Suppose that $T$ is generalized pseudo-hyperbolic, and let $D, \lambda>0$ and $C\in H(X)$ be as in Definition~\ref{autgph}. Let $T_n\equiv T$ be the same constant sequence in $GL(X)$ as in the previous part of the proof.
The exponential estimates in Definition~\ref{autgph}, together with~\eqref{autevol}, imply that the constant sequence $(T_n)_{n\in\mathbb Z}$ satisfies Definition~\ref{gphdef} with $C_n:=C$ for all $n\in\mathbb Z$, i.e., it is generalized pseudo-hyperbolic. It follows that the corresponding Green kernel~\eqref{Gmndef} has the form
\begin{equation*}
G_{m,n}:=
\begin{cases}
T^{m-n-1}C, & m\ge n+1,\\[1mm]
-T^{m-n-1}(I-C), & m<n+1.
\end{cases}
\end{equation*}
This implies that
\begin{equation}\label{autgreen}
G_{m,n}=G_{m-n,0},\qquad m,n\in\mathbb Z.
\end{equation}
By Theorem~\ref{thm1}, the generalized pseudo-hyperbolicity of the constant sequence $(T_n)_{n\in\mathbb Z}$ implies its topological stability. More precisely, given $\epsilon>0$, choose $\delta>0$ as in the proof of Theorem~\ref{thm1}, part $(ii)\Rightarrow(iii)$, for the constant sequence $T_n\equiv T$. Let $g\in Homeo(X)$ satisfy
\begin{equation*}
\|g-T\|_\infty\leq\delta,
\end{equation*}
and set $g_n:=g$ for all $n\in\mathbb Z$. Following the construction in the proof of Theorem~\ref{thm1}, we may choose 
\begin{equation*}
h_n=I+H_n,\qquad n\in\mathbb Z,
\end{equation*}
where $H_n$ are continuous self-maps of~$X$ defined by~\eqref{Hndef} with~$f_n$ as in~\eqref{fndef}. In the present autonomous setting,
\begin{equation*}
f_n=f:=T-g
\qquad\text{and}\qquad
g(m,n)=g^{m-n},
\qquad m,n\in\mathbb Z.
\end{equation*}
This, together with~\eqref{autgreen}, implies that~\eqref{Hndef} can be written in the form
\begin{equation*}
H_n(x):=\sum_{m\in\mathbb Z}
G_{n-m,0}\bigl(f(g^{m-n}(x))\bigr)=\sum_{j\in\mathbb Z}
G_{j,0}\bigl(f(g^{-j}(x))\bigr)
\end{equation*}
for $x\in X$. Thus, $H_n$ is independent of $n$, and hence so is $h_n=I+H_n$. Writing $h:=h_0$, by~\eqref{topc1} and~\eqref{topc2},  we  obtain
\begin{equation*}
h\circ g=T\circ h
\qquad\text{and}\qquad
\|h-I\|_\infty\leq\epsilon.
\end{equation*}
Since $g\in Homeo(X)$ with $\|g-T\|_\infty\leq\delta$ was arbitrary, it follows that $T$ is topologically stable.
\medskip

\textit{(iv) $\Rightarrow$ (i).} By \cite[Corollary~15]{BCDHP}, if $T\in GL(X)$ is topologically stable, then~$T$ has the shadowing property.
\end{proof}

\begin{proof}[Proof of Theorem~\ref{thm3}]
By Proposition~\ref{prop2}, the shadowing property of $(T_n)_{n\in\mathbb Z}$ implies that $\mathcal L\colon\ell'_\infty(X)\to\ell_\infty(X)$ is surjective. Therefore, by the Open Mapping Theorem, there exists $K>0$ such that for every $z\in\ell_\infty(X)$ there exists $y\in\ell'_\infty(X)$ satisfying
\begin{equation}\label{omc}
\mathcal Ly=z \quad \text{and} \quad 
 \|y\|_\infty'\le K\|z\|_\infty.
 \end{equation}
 Choose $\eta\in(0,1/K)$. Suppose that $(f_n)_{n\in\mathbb Z}$ is a sequence of Lipschitz self-maps of~$X$ such that
 \begin{equation*}
 \operatorname{Lip}(f_n)<\eta\qquad\text{and}\qquad F_n=T_n+f_n\in Homeo(X)\quad\text{for $n\in\mathbb Z$}.
 \end{equation*}
 Let $\epsilon>0$ and choose
\begin{equation*}
\delta:=\frac{1-\eta K}{K}\,\epsilon>0.
\end{equation*}
Let $x=(x_n)_{n\in\mathbb Z}$ be a $\delta$-pseudo-orbit of~$(F_n)_{n\in\mathbb Z}$, i.e., $\|d\|_\infty\leq\delta$, where 
\begin{equation*}
d_n:=x_{n+1}-T_nx_n-f_n(x_n),\qquad n\in\mathbb Z.
\end{equation*}
 Define
\begin{equation*}
(\mathcal H(z))_n:=f_n(x_n+z_n)-f_n(x_n),\qquad n\in\mathbb Z,
\end{equation*}
for $z=(z_n)_{n\in\mathbb Z}\in\ell_\infty(X)$.
Then
\begin{equation*}
\|\mathcal H(z)\|_\infty\leq \left (\sup_{n\in\mathbb Z}\operatorname{Lip}(f_n)\right )\|z\|_\infty\le \eta\|z\|_\infty, \qquad z\in\ell_\infty(X),
\end{equation*}
and
\begin{equation}\label{Hlip}
\|\mathcal H(z)-\mathcal H(w)\|_\infty\le \eta\|z-w\|_\infty,
\qquad z,w\in\ell_\infty(X).
\end{equation}
We construct inductively a sequence $z^{(k)}\in\ell'_\infty(X)$, $k\in\mathbb N_0$. Set $z^{(0)}=0$. Suppose that $z^{(k)}\in\ell'_\infty(X)$ has been defined for some $k\in\mathbb N_0$. Since $\mathcal H(z^{(k)})$,~$d$ and $\mathcal Lz^{(k)}$ belong to~$\ell_\infty(X)$,
\begin{equation}\label{baseq}
r^{(k)}:=\mathcal H(z^{(k)})-d-\mathcal Lz^{(k)}\in\ell_\infty(X).
\end{equation}
By~\eqref{omc}, we can choose $h^{(k)}\in\ell_\infty'(X)$ such that
\begin{equation}\label{hkest}
\mathcal Lh^{(k)}=r^{(k)}\qquad\text{and}\qquad
\|h^{(k)}\|_\infty'\le K\|r^{(k)}\|_\infty.
\end{equation}
Define $z^{(k+1)}:=z^{(k)}+h^{(k)}\in\ell'_\infty(X)$. Then
\begin{equation*}
r^{(k+1)}=\mathcal H(z^{(k+1)})-\mathcal H(z^{(k)}),\qquad k\in\mathbb N_0.
\end{equation*}
Consequently, by~\eqref{Hlip} and~\eqref{hkest},
\begin{equation*}
\|r^{(k+1)}\|_\infty
\le \eta\|h^{(k)}\|_\infty\le\eta\|h^{(k)}\|'_\infty
\le \eta K\|r^{(k)}\|_\infty,\qquad k\in\mathbb N_0.
\end{equation*}
Since $r^{(0)}=-d$ (see~\eqref{baseq}), we obtain
\begin{equation}\label{rkest}
\|r^{(k)}\|_\infty\le (\eta K)^k\delta,\qquad k\in\mathbb N_0.
\end{equation}
By~\eqref{hkest} and~\eqref{rkest},
\begin{equation}\label{zkdif}
\|z^{(k+1)}-z^{(k)}\|_\infty'
=\|h^{(k)}\|_\infty'
\le K(\eta K)^k\delta,\qquad k\in\mathbb N_0.
\end{equation}
Since $\eta K<1$, this implies that $(z^{(k)})_{k\in\mathbb N_0}$ is a Cauchy sequence in the Banach space $(\ell_\infty'(X),\|\cdot\|_\infty')$. Consequently, there exists $z\in\ell_\infty'(X)$ such that $z^{(k)}\to z$ in the graph norm $\|\cdot\|'_\infty$. 
In particular,
\begin{equation*}
z^{(k)}\to z\quad\text{and}\quad \mathcal Lz^{(k)}\to\mathcal Lz
\quad\text{in }\ell_\infty(X).
\end{equation*}
Hence
$z=\sum_{k=0}^\infty(z^{(k+1)}-z^{(k)})$,
and~\eqref{zkdif} yields
\begin{equation}\label{zest}
	\|z\|'_\infty\leq 
	\sum_{k=0}^\infty\|z^{(k+1)}-z^{(k)}\|'_\infty
	\leq K\delta\sum_{k=0}^\infty(\eta K)^k=\frac{K\delta}{1-\eta K}=\epsilon.
\end{equation}
Since $r^{(k)}\to0$ by~\eqref{rkest}, and since
\begin{equation*}
z^{(k)}\to z,\qquad
\mathcal Lz^{(k)}\to\mathcal Lz,\qquad\text{and}\qquad
\mathcal H(z^{(k)})\to\mathcal H(z)
\end{equation*}
in $\ell_\infty(X)$, the last convergence being a consequence of~\eqref{Hlip},
letting $k\to\infty$ in~\eqref{baseq} yields
\begin{equation*}
	\mathcal H(z)=d+\mathcal Lz.
\end{equation*}
Equivalently,
\begin{equation*}
f_n(x_n+z_n)-f_n(x_n)=d_n+z_{n+1}-T_n(z_n),\qquad n\in\mathbb Z.
\end{equation*}
By the definition of $d_n$, this implies for $n\in\mathbb Z$,
	\begin{equation*}
f_n(x_n+z_n)-f_n(x_n)=x_{n+1}-T_n(x_n)-f_n(x_n)+z_{n+1}-T_n(z_n).
\end{equation*}
Hence
\begin{equation*}
x_{n+1}+z_{n+1}
=T_n(x_n+z_n)+f_n(x_n+z_n),
\qquad n\in\mathbb Z.
\end{equation*}
Therefore, setting
\begin{equation*}
y_n:=x_n+z_n,\qquad n\in\mathbb Z,
\end{equation*}
we obtain
\begin{equation*}
y_{n+1}=F_n(y_n),\qquad n\in\mathbb Z.
\end{equation*}
Thus, $y=(y_n)_{n\in\mathbb Z}$ is an exact bi-infinite trajectory of $(F_n)_{n\in\mathbb Z}$. Since each $F_n$, $n\in\mathbb Z$, is a homeomorphism,
\begin{equation*}
y_n=F(n,0)(y_0),\qquad n\in\mathbb Z.
\end{equation*}
Finally, \eqref{zest} implies that
\begin{equation*}
\sup_{n\in\mathbb Z}\|x_n-F(n,0)(y_0)\|
=\sup_{n\in\mathbb Z}\|x_n-y_n\|=\sup_{n\in\mathbb Z}\|z_n\|
=\|z\|_\infty
\le \|z\|_\infty'
\le \epsilon.
\end{equation*}
Therefore, $(F_n)_{n\in\mathbb Z}$ has the shadowing property.
\end{proof}

\begin{proof}[Proof of Theorem~\ref{thm4}]
By Theorem~\ref{thm1}, the topological stability of $(T_n)_{n\in\mathbb Z}$ implies that $(T_n)_{n\in\mathbb Z}$ is generalized pseudo-hyperbolic. Let $G_{m,n}\in H(X)$, $m,n\in\mathbb Z$, be the corresponding Green family given by~\eqref{Gmndef}. Thus,
\begin{equation}\label{green_ident}
G_{m+1,n}-T_mG_{m,n}=
\begin{cases}
I,\quad & m=n,\\
0,\quad & m\neq n,
\end{cases}
\end{equation}
and there exist constants $D,\lambda>0$ such that
\begin{equation}\label{green_est}
\|G_{m,n}\|\le De^{-\lambda|m-n-1|},
\qquad m,n\in\mathbb Z.
\end{equation}
Set
\begin{equation}\label{mudef}
K_0:=D\sum_{j\in\mathbb Z}e^{-\lambda|j|}
=D\frac{1+e^{-\lambda}}{1-e^{-\lambda}}>0.
\end{equation}
Choose
\begin{equation}\label{eta_choice}
0<\eta<\frac1{K_0}.
\end{equation}
Suppose that $(f_n)_{n\in\mathbb Z}$ is a sequence of Lipschitz self-maps of $X$ such that
\begin{equation*}
\operatorname{Lip}(f_n)<\eta\qquad\text{and}\qquad F_n:=T_n+f_n\in Homeo(X)
\quad\text{for all $n\in\mathbb Z$}.
\end{equation*}
We prove that $(F_n)_{n\in\mathbb Z}$ is topologically stable.
Let $\epsilon>0$ and choose
\begin{equation}\label{delta_choice}
\delta:=\frac{1-\eta K_0}{K_0}\,\epsilon>0.
\end{equation}
Let $(g_n)_{n\in\mathbb Z}$ be an arbitrary sequence in $Homeo(X)$ satisfying
\begin{equation}\label{g_close}
\sup_{n\in\mathbb Z}\|F_n-g_n\|_\infty\le\delta.
\end{equation}
Set
\begin{equation}\label{pndef}
p_n:=F_n-g_n,
\qquad n\in\mathbb Z.
\end{equation}
Then
\begin{equation}\label{pn_est}
\sup_{n\in\mathbb Z}\|p_n\|_\infty\le\delta.
\end{equation}
Let $C_b(X)$ denote the space of bounded continuous self-maps of~$X$ equipped with the norm $\|\cdot\|_\infty$. Since $C_b(X)$ is a Banach space, the space $\ell_\infty(C_b(X))=\ell_\infty(\mathbb Z,C_b(X))$ is also a Banach space. Recall that the norm on $\ell_\infty(C_b(X))$ is given by
\begin{equation*}
\|q\|_\infty:=\sup_{n\in\mathbb Z}\|q_n\|_\infty<\infty\qquad\text{for $q=(q_n)_{n\in\mathbb Z}\in\ell_\infty(C_b(X))$}.
\end{equation*}
Given $q=(q_n)_{n\in\mathbb Z}\in\ell_\infty(C_b(X))$, define
\begin{equation}\label{Rg_def}
(\mathcal R_gq)_n(x)
:=\sum_{m\in\mathbb Z}G_{n,m}\bigl(q_m(g(m,n)(x))\bigr)
\end{equation}
for $n\in\mathbb Z$ and $x\in X$, where $g(m,n)$ denotes the evolution family associated with $(g_n)_{n\in\mathbb Z}$.
By~\eqref{green_est},
\begin{equation*}
\begin{aligned}
\|G_{n,m}(q_m(g(m,n)(x)))\|
&\le De^{-\lambda|n-m-1|}\,
\|q_m(g(m,n)(x))\|\\
&\le De^{-\lambda|n-m-1|}\|q\|_\infty.
\end{aligned}
\end{equation*}
Since
\begin{equation*}
\sum_{m\in\mathbb Z}e^{-\lambda|n-m-1|}
=\sum_{j\in\mathbb Z}e^{-\lambda|j|}<\infty,
\qquad n\in\mathbb Z,
\end{equation*}
the series in~\eqref{Rg_def} converges uniformly in $x\in X$, for every $n\in\mathbb Z$. Hence each $(\mathcal R_gq)_n$, $n\in\mathbb Z$, is continuous and
\begin{equation*}
\|\mathcal R_gq\|_\infty
:=\sup_{n\in\mathbb Z}\|(\mathcal R_gq)_n\|_\infty
\leq D\|q\|_\infty\sum_{m\in\mathbb Z}e^{-\lambda|n-m-1|}
= D\|q\|_\infty\sum_{j\in\mathbb Z}e^{-\lambda|j|}.
\end{equation*}
Hence, by~\eqref{mudef},
\begin{equation}\label{Rg_est}
	\|\mathcal R_gq\|_\infty\leq K_0\|q\|_\infty. 
\end{equation}
Thus, $\mathcal R_g$ maps $\ell_\infty(C_b(X))$ into itself.
Moreover, by the evolution property,
\begin{equation*}
g(m,n+1)\circ g_n=g(m,n),
\qquad m,n\in\mathbb Z.
\end{equation*}
Therefore, using~\eqref{Rg_def}, we obtain for $n\in\mathbb Z$ and $x\in X$,
\begin{equation*}
\begin{aligned}
&(\mathcal R_gq)_{n+1}(g_n(x))-T_n(\mathcal R_gq)_n(x)\\
&=\sum_{m\in\mathbb Z}G_{n+1,m}
\bigl(q_m(g(m,n+1)(g_n(x)))\bigr)
-T_n\sum_{m\in\mathbb Z}G_{n,m}
\bigl(q_m(g(m,n)(x))\bigr)\\
&=\sum_{m\in\mathbb Z}
\bigl(G_{n+1,m}-T_nG_{n,m}\bigr)
\bigl(q_m(g(m,n)(x))\bigr).
\end{aligned}
\end{equation*}
By~\eqref{green_ident}, the only nonzero term in the last sum corresponds to $m=n$. Since $g(n,n)=I$, it follows that for every $x\in X$,
\begin{equation}\label{Rg_rightinverse}
(\mathcal R_gq)_{n+1}(g_n(x))-T_n(\mathcal R_gq)_n(x)=q_n(x),
\qquad n\in\mathbb Z.
\end{equation}
For $h=(h_n)_{n\in\mathbb Z}\in\ell_\infty(C_b(X))$, define
\begin{equation}\label{N_def}
(\mathcal N(h))_n(x):=f_n(x+h_n(x))-f_n(x),
\qquad n\in\mathbb Z,\quad x\in X.
\end{equation}
Since $\operatorname{Lip}(f_n)<\eta$ for all $n\in\mathbb Z$, it follows that
\begin{equation}\label{N_bound}
\|\mathcal N(h)\|_\infty\le\eta\|h\|_\infty,\qquad h\in\ell_\infty(C_b(X)).
\end{equation}
Similarly,
\begin{equation}\label{N_lip}
\|\mathcal N(h)-\mathcal N(k)\|_\infty
\le\eta\|h-k\|_\infty,\qquad h,\,k\in\ell_\infty(C_b(X)).
\end{equation}
Thus, $\mathcal N$ is a Lipschitz self-map of $\ell_\infty(C_b(X))$.

For $h=(h_n)_{n\in\mathbb Z}\in\ell_\infty(C_b(X))$, define 
\begin{equation*}
(\mathcal L_gh)_n(x):=h_{n+1}(g_n(x))-T_nh_n(x),
\qquad n\in\mathbb Z,\quad x\in X.
\end{equation*}
We construct inductively a sequence
$h^{(k)}=(h_n^{(k)})_{n\in\mathbb Z}$, $k\in\mathbb N_0$,
in $\ell_\infty(C_b(X))$ such that
$\mathcal L_g h^{(k)}\in\ell_\infty(C_b(X))$ for every $k\in\mathbb N_0$.
Set
$h^{(0)}=0$.
Suppose that $h^{(k)}\in\ell_\infty(C_b(X))$ has been defined such that $\mathcal L_gh^{(k)}\in\ell_\infty(C_b(X))$ for some $k\in\mathbb N_0$. Set
\begin{equation}\label{res_def_general}
r^{(k)}:=p+\mathcal N(h^{(k)})-\mathcal L_gh^{(k)}\in\ell_\infty(C_b(X)),
\end{equation}
where $p=(p_n)_{n\in\mathbb Z}\in\ell_\infty(C_b(X))$ is defined by~\eqref{pndef}.
Define
\begin{equation}\label{ak_def}
a^{(k)}:=\mathcal R_gr^{(k)}\in\ell_\infty(C_b(X))
\end{equation}
and
\begin{equation}\label{hk_step}
h^{(k+1)}:=h^{(k)}+a^{(k)}\in\ell_\infty(C_b(X)).
\end{equation}
By~\eqref{Rg_rightinverse}, $\mathcal L_g\mathcal R_g q=q$ for all $q\in\ell_\infty(C_b(X))$. Hence
\begin{equation*}
\mathcal L_ga^{(k)}=\mathcal L_g\mathcal R_g r^{(k)}=r^{(k)}\in\ell_\infty(C_b(X)).
\end{equation*}
Since $\mathcal L_g$ is linear, this and~\eqref{hk_step} imply that
\begin{equation}\label{Lga_r}
\mathcal L_gh^{(k+1)}=\mathcal L_gh^{(k)}+r^{(k)}\in\ell_\infty(C_b(X)).
\end{equation} 
From \eqref{res_def_general} and~\eqref{Lga_r}, we obtain
\begin{equation*}
\begin{aligned}
r^{(k+1)}
&=p+\mathcal N(h^{(k+1)})-\mathcal L_gh^{(k+1)}\\
&=p+\mathcal N(h^{(k+1)})-\mathcal L_gh^{(k)}-r^{(k)}\\
&=\mathcal N(h^{(k+1)})-\mathcal N(h^{(k)}).
\end{aligned}
\end{equation*}
Consequently,~\eqref{N_lip}, \eqref{hk_step}, ~\eqref{ak_def} and \eqref{Rg_est}, yield
\begin{equation*}
\|r^{(k+1)}\|_\infty
\le\eta\|h^{(k+1)}-h^{(k)}\|_\infty
=\eta\|a^{(k)}\|_\infty
=\eta\|\mathcal R_g r^{(k)}\|_\infty
\le\eta K_0\|r^{(k)}\|_\infty
\end{equation*}
for $k\in\mathbb N_0$.
Since $r^{(0)}=p$, from the last inequality and~\eqref{pn_est}, we obtain
\begin{equation}\label{rk_est}
\|r^{(k)}\|_\infty\le(\eta K_0)^k\|r^{(0)}\|_\infty=(\eta K_0)^k\|p\|_\infty
\le(\eta K_0)^k\delta,
\qquad k\in\mathbb N_0.
\end{equation}
Furthermore, by~\eqref{hk_step},~\eqref{ak_def},~\eqref{Rg_est} and~\eqref{rk_est},
\begin{equation}\label{hk_dif}
\|h^{(k+1)}-h^{(k)}\|_\infty
=\|a^{(k)}\|_\infty=\|\mathcal R_g r^{(k)}\|_\infty
\le K_0\|r^{(k)}\|_\infty
\le K_0(\eta K_0)^k\delta.
\end{equation}
Since $\eta K_0<1$, it follows that $(h^{(k)})_{k\in\mathbb N_0}$ is a Cauchy sequence in the Banach space $\ell_\infty(C_b(X))$. Thus, there exists $h=(h_n)_{n\in\mathbb Z}\in\ell_\infty(C_b(X))$ such that
$h^{(k)}\to h$. Moreover, by~\eqref{hk_dif},
\begin{equation*}
\|h\|_\infty
\le\sum_{k=0}^\infty
\|h^{(k+1)}-h^{(k)}\|_\infty
\le K_0\delta\sum_{k=0}^\infty(\eta K_0)^k
=\frac{K_0\delta}{1-\eta K_0}
=\epsilon,
\end{equation*}
where the last equality follows from~\eqref{delta_choice}. Hence
\begin{equation}\label{h_est}
\sup_{n\in\mathbb Z}\|h_n\|_\infty\le\epsilon.
\end{equation}
By~\eqref{rk_est}, $r^{(k)}\to0$, while $h^{(k)}\to h$ and~\eqref{N_lip} imply that $\mathcal N(h^{(k)})\to\mathcal N(h)$ in $\ell_\infty(C_b(X))$. Since the convergence in $\ell_\infty(C_b(X))$ is equivalent to uniform convergence in~$n\in\mathbb Z$ and $x\in X$, we have
\begin{equation*}
h_n^{(k)}(x)\rightarrow h_n(x)\qquad\text{and}\qquad h_{n+1}^{(k)}(g_n(x))\rightarrow h_{n+1}(g_n(x))\qquad\text{as $k\to\infty$}
\end{equation*}
for all $n\in\mathbb Z$ and $x\in X$. Moreover, by the continuity of~$T_n$,
\begin{equation*}
T_n h^{(k)}_n(x)\to T_n h_n(x)\qquad\text{as $k\to\infty$}
\end{equation*}
for all $n\in\mathbb Z$ and $x\in X$.
Therefore, passing to the limit in~\eqref{res_def_general}, pointwise for the fixed $n$ and $x$, yields
\begin{equation*}
0=p_n(x)+(\mathcal N(h))_n(x)
-h_{n+1}(g_n(x))+T_nh_n(x).
\end{equation*}
This, combined with~\eqref{N_def}, implies for $n\in\mathbb Z$ and $x\in X$,
\begin{equation}\label{cohom_eq_2}
h_{n+1}(g_n(x))-T_nh_n(x)
=p_n(x)+f_n(x+h_n(x))-f_n(x).
\end{equation}
Finally, define
\begin{equation}\label{Hn_def}
H_n:=I+h_n,
\qquad n\in\mathbb Z.
\end{equation}
Each $H_n$, $n\in\mathbb Z$, is continuous. Using~\eqref{cohom_eq_2}, \eqref{pndef}, and $F_n=T_n+f_n$, we obtain for $n\in\mathbb Z$ and $x\in X$,
\begin{equation*}
\begin{aligned}
H_{n+1}(g_n(x))
&=g_n(x)+h_{n+1}(g_n(x))\\
&=g_n(x)+T_nh_n(x)+p_n(x)
+f_n(x+h_n(x))-f_n(x)\\
&=g_n(x)+T_nh_n(x)+F_n(x)-g_n(x)
+f_n(x+h_n(x))-f_n(x)\\
&=T_nx+T_nh_n(x)+f_n(x+h_n(x))\\
&=F_n(x+h_n(x))\\
&=F_n(H_n(x)).
\end{aligned}
\end{equation*}
Hence
\begin{equation}\label{semiconj_final}
H_{n+1}\circ g_n=F_n\circ H_n,
\qquad n\in\mathbb Z.
\end{equation}
Moreover, by~\eqref{h_est},
\begin{equation*}
\sup_{n\in\mathbb Z}\|H_n-I\|_\infty
=\sup_{n\in\mathbb Z}\|h_n\|_\infty
\le\epsilon.
\end{equation*}
Since $(g_n)_{n\in\mathbb Z}$ was an arbitrary sequence in $Homeo(X)$ satisfying~\eqref{g_close}, it follows that $(F_n)_{n\in\mathbb Z}$ is topologically stable.
\end{proof}

\begin{proof}[Proof of Corollary~\ref{cor1}]
Let $\eta>0$ be given by Theorem~\ref{thm3}, and set
\begin{equation*}
M:=\sup_{n\in\mathbb Z}\|T_n^{-1}\|<\infty.
\end{equation*}
Choose
\begin{equation*}
0<\tilde\eta<\min\left\{\eta,\frac1M\right\}.
\end{equation*}
We show that $\tilde\eta$ has the required property.
Suppose that $(f_n)_{n\in\mathbb Z}$ is a sequence of Lipschitz self-maps of~$X$ such that
$\operatorname{Lip}(f_n)<\tilde\eta$ for all $n\in\mathbb Z$.
Then, for every $n\in\mathbb Z$,
\begin{equation*}
\|T_n^{-1}\|\operatorname{Lip}(f_n)
\le M\operatorname{Lip}(f_n)<M\tilde\eta<1.
\end{equation*}
Hence, by Remark~\ref{rem2},
\begin{equation*}
F_n:=T_n+f_n\in Homeo(X),\qquad n\in\mathbb Z.
\end{equation*}
Since $\operatorname{Lip}(f_n)<\tilde\eta<\eta$ for every $n\in\mathbb Z$, Theorem~\ref{thm3} implies that $(F_n)_{n\in\mathbb Z}$ has the shadowing property.
\end{proof}	

\begin{proof}[Proof of Corollary~\ref{cor2}]
The proof is identical to that of Corollary~\ref{cor1}. One chooses $\tilde\eta>0$ small enough so that Remark~\ref{rem2} guarantees that $F_n:=T_n+f_n\in Homeo(X)$ for all $n\in\mathbb Z$, and then applies Theorem~\ref{thm4} instead of Theorem~\ref{thm3}.
\end{proof}

\begin{proof}[Proof of Theorem~\ref{thm5}]
Let $(T_n)_{n\in\mathbb Z}$ be the constant sequence defined by
$T_n:=T$ for all $n\in\mathbb Z$.
The shadowing property of~$T$ implies that the constant sequence
$(T_n)_{n\in\mathbb Z}$ also has the shadowing property. Let
$\eta>0$ be given by Theorem~\ref{thm3}. Choose
\begin{equation*}
0<\tilde\eta<\min\left\{\eta,\frac1{\|T^{-1}\|}\right\}.
\end{equation*}
Let $f$ be a Lipschitz self-map of~$X$ such that
$\operatorname{Lip}(f)<\tilde\eta$.
Then
\begin{equation*}
\|T^{-1}\|\operatorname{Lip}(f)<1,
\end{equation*}
and hence, by Remark~\ref{rem2},
\begin{equation*}
F:=T+f\in Homeo(X).
\end{equation*}
Now consider the constant sequence
$f_n:=f$ for all $n\in\mathbb Z$.
Since
\begin{equation*}
\operatorname{Lip}(f_n)=\operatorname{Lip}(f)<\tilde\eta<\eta
\qquad\text{for all $n\in\mathbb Z$},
\end{equation*}
Theorem~\ref{thm3} implies that the constant sequence
$F_n:=T_n+f_n=F$, $n\in\mathbb Z$
has the shadowing property. By Definitions~\ref{autshad} and~\ref{nonautshad}, the shadowing property of~$F$ is equivalent to that of the associated constant sequence $(F_n)_{n\in\mathbb Z}$.
 Therefore, $F$ has the shadowing property.
\end{proof}

\begin{proof}[Proof of Theorem~\ref{thm6}]
By Theorem~\ref{thm2}, the topological stability of~$T\in GL(X)$ implies that $T$ is generalized pseudo-hyperbolic. Let $C\in H(X)$ and $D,\lambda>0$ be as in Definition~\ref{autgph}. Consider the constant sequence $T_n:=T$ for all $n\in\mathbb Z$.
As in the proof of Theorem~\ref{thm2}, this sequence is generalized pseudo-hyperbolic with $C_n:=C$ for all $n\in\mathbb Z$. Its Green family
\begin{equation*}
G_{m,n}=\begin{cases}
T^{m-n-1}C, & m\ge n+1,\\[1mm]
-T^{m-n-1}(I-C), & m<n+1,
\end{cases}
\end{equation*}
satisfies $G_{m,n}=G_{m-n,0}$ for all $m,n\in\mathbb Z$.
Let
\begin{equation*}
K_0:=D\sum_{j\in\mathbb Z}e^{-\lambda|j|}
=D\frac{1+e^{-\lambda}}{1-e^{-\lambda}}.
\end{equation*}
Choose
\begin{equation*}
0<\eta<\min\left\{\frac1{K_0},\frac1{\|T^{-1}\|}\right\}.
\end{equation*}
Let $f$ be a Lipschitz self-map of~$X$ satisfying $\operatorname{Lip}(f)<\eta$. Since
\begin{equation*}
\|T^{-1}\|\operatorname{Lip}(f)<1,
\end{equation*}
as noted in Remark~\ref{rem2}, $F:=T+f\in Homeo(X)$.
We prove that $F$ is topologically stable. Let $\epsilon>0$. Choose
\begin{equation*}
\delta:=\frac{1-\eta K_0}{K_0}\,\epsilon>0.
\end{equation*}
Let $g\in Homeo(X)$ satisfy $\|F-g\|_\infty\le\delta$. Consider the constant sequences
\begin{equation*}
f_n:=f\qquad\text{and}\qquad g_n:=g,\qquad n\in\mathbb Z.
\end{equation*}
We apply the construction from the proof of Theorem~\ref{thm4}. In the present autonomous setting,
\begin{equation*}
p_n:=F-g=:p,\qquad n\in\mathbb Z,
\end{equation*}
is independent of $n$, and the initial approximation $h^{(0)}=0$ is also independent of $n$.
We claim that, for every $k\in\mathbb N_0$, the sequence $h^{(k)}=(h^{(k)}_n)_{n\in\mathbb Z}$ 
constructed in the proof of Theorem~\ref{thm4} is independent of~$n$. We prove this by induction on $k$. The claim is obvious for $k=0$.
Suppose that, for some $k\in\mathbb N_0$, $h^{(k)}_n=h^{(k)}_0$ for all $n\in\mathbb Z$. Then, for every $n\in\mathbb Z$ and $x\in X$,
\begin{equation*}
(\mathcal N(h^{(k)}))_n(x)
=
f(x+h_0^{(k)}(x))-f(x),
\end{equation*}
and
\begin{equation*}
(\mathcal L_g h^{(k)})_n(x)
=
h_0^{(k)}(g(x))-T h_0^{(k)}(x).
\end{equation*}
Hence both $\mathcal N(h^{(k)})$ and $\mathcal L_g h^{(k)}$ are independent of $n$. Therefore,
\begin{equation*}
r^{(k)}
=
p+\mathcal N(h^{(k)})-\mathcal L_g h^{(k)}
\end{equation*}
is also independent of $n$.
Now let $q=(q_n)_{n\in\mathbb Z}\in\ell_\infty(C_b(X))$ be independent of $n$, so that $q_n=q_0$ for all $n\in\mathbb Z$. Using the definition of $\mathcal R_g$, the translation invariance of the Green family, and the relation $g(m,n)=g^{m-n}$, we have
\begin{equation*}
(\mathcal R_gq)_n(x)
=\sum_{m\in\mathbb Z}G_{n-m,0}\bigl(q_0(g^{m-n}(x))\bigr)
=\sum_{j\in\mathbb Z}G_{j,0}\bigl(q_0(g^{-j}(x))\bigr).
\end{equation*}
Thus, $\mathcal R_g q$ is independent of~$n$. Applying this to $q=r^{(k)}$, we conclude that $a^{(k)}:=\mathcal R_g r^{(k)}$
is independent of~$n$. Since $h^{(k+1)}=h^{(k)}+a^{(k)}$, it follows that $h^{(k+1)}$ is independent of~$n$. This proves the claim. 

By the proof of Theorem~\ref{thm4}, there exists $h=(h_n)_{n\in\mathbb Z}\in\ell_\infty(C_b(X))$ such that $h^{(k)}\to h$ in $\ell_\infty(C_b(X))$ and
\begin{equation*}
\sup_{n\in\mathbb Z}\|h_n\|_\infty\le\epsilon.
\end{equation*}
Since every $h^{(k)}$ is independent of~$n$ and $h^{(k)}\to h$ in $\ell_\infty(C_b(X))$,  the limit~$h$ is also independent of~$n$. Therefore, $h_n=h_0\in C_b(X)$ for all $n\in\mathbb Z$. Set $H:=I+h_0$.
The semiconjugacy relation~\eqref{semiconj_final} obtained in the proof of Theorem~\ref{thm4}
reduces to
\begin{equation*}
H\circ g=F\circ H.
\end{equation*}
Moreover,
\begin{equation*}
\|H-I\|_\infty=\|h_0\|_\infty\le\epsilon.
\end{equation*}
Since $g\in Homeo(X)$ with $\|F-g\|_\infty\le\delta$ was arbitrary, $F$ is topologically stable.
\end{proof}

	\section*{Acknowledgements}  
This paper has been funded by the European Union – NextGenerationEU-StatDin\-MatAn-uniri-iz-25-108 (Davor Dragi\v cevi\' c).


\begin{thebibliography}{99}

\bibitem{BD}
L. Backes and D. Dragi\v cevi\' c, \emph{Shadowing for infinite dimensional dynamics and exponential trichotomies}, Proc. Roy. Soc. Edinburgh Sect. A \textbf{151} (2021), 863--884.

\bibitem{B}
N. C. Bernardes Jr., \emph{On the stability of generalized hyperbolic operators}, Rev. Mat. Complut. (2026), DOI: 10.1007/s13163-026-00569-5.

\bibitem{BCDHP}
N. C. Bernardes Jr., B. M. Caraballo, U. B. Darji, V. V. F\'avaro and A. Peris, \emph{Generalized hyperbolicity, stability and expansivity for operators on locally convex spaces}, J. Funct. Anal. \textbf{288} (2025), no.~2, Article 110696.

\bibitem{BCDMP}
N. C. Bernardes Jr., P. R. Cirilo, U. B. Darji, A. Messaoudi and E. R. Pujals, \emph{Expansivity and shadowing in linear dynamics}, J. Math. Anal. Appl. \textbf{461} (2018), no.~1, 796--816.



\bibitem{BM}
N. C. Bernardes Jr. and A. Messaoudi, \emph{A generalized Grobman--Hartman theorem}, Proc. Amer. Math. Soc. \textbf{148} (2020), no.~10, 4351--4360.

\bibitem{BM2}
N. C. Bernardes Jr. and A. Messaoudi, \emph{Shadowing and structural stability for operators}, Ergodic Theory Dynam. Systems \textbf{41} (2021), no.~4, 961--980.

\bibitem{BP}
N. C. Bernardes Jr. and A. Peris, \emph{On shadowing and chain recurrence in linear dynamics}, Adv. Math. \textbf{441} (2024), Article 109539.

\bibitem{CGP}
P. R. Cirilo, B. Gollobit and E. R. Pujals, \emph{Dynamics of generalized hyperbolic linear operators}, Adv. Math. \textbf{387} (2021), Article 107830.

\bibitem{DM}
E. D'Aniello, U. B. Darji and M. Maiuriello, \emph{Generalized hyperbolicity and shadowing in $L^p$ spaces}, J. Differential Equations \textbf{298} (2021), 68--94.

\bibitem{DDM}
E. D'Aniello, U. B. Darji and M. Maiuriello, \emph{Shift-like operators on $L^p(X)$}, J. Math. Anal. Appl. \textbf{515} (2022), no.~1, Article 126393.

\bibitem{DP}
D. Dragi\v cevi\' c and M. Pituk, \emph{Duality between shadowing and uniform expansivity in linear dynamics}, Trans. Amer. Math. Soc. (2026), DOI: 10.1090/tran/9879.

\bibitem{H}
D. Henry, \emph{Geometric Theory of Semilinear Parabolic Equations}, Lecture Notes in Mathematics, vol.~840, Springer, Berlin, 1981.

\bibitem{HP}
E. Hille and R. S. Phillips, \emph{Functional Analysis and Semi-Groups}, American Mathematical Society, Providence, RI, 1957.

\bibitem{LMN}
K. Lee, C. A. Morales and N. Nguyen, \emph{Topologically stable linear operators}, Bull. Soc. Math. France \textbf{151} (2023), no.~4, 647--657.

\bibitem{MTS}
A. Messaoudi, J. Tofanin Neto, M. Saavedra and I. Tsokanos, \emph{On generalized hyperbolicity, stability, and shadowing for linear operators}, arXiv:2608.17021, 2026. \url{https://arxiv.org/abs/2608.17021}

\bibitem{Mul}
V. M\"uller, \emph{Spectral Theory of Linear Operators and Spectral Systems in Banach Algebras}, second edition, Birkh\"auser, Basel, 2007.

\bibitem{N}
N. Nguyen, \emph{Stability of limit expansive systems},
J. Math. Anal. Appl. \textbf{547} (2025), Article 129335.

\bibitem{PS}
M. Pituk, \emph{Spectral characterization of shadowing for linear operators on Hilbert spaces}, Math. Ann., in press, arXiv:2511.12272, 2025.
 \url{https://arxiv.org/abs/2511.12272}

\bibitem{P}
M. Pituk, \emph{Resolving the generalized hyperbolicity conjecture for shadowing}, arXiv:2608.19499, 2026. 
\url{https://arxiv.org/abs/2608.19499}

\end{thebibliography}
\end{document}